\documentclass[12pt]{amsart}
\usepackage[authoryear]{natbib} 
\usepackage[english]{babel}
\usepackage{amsmath}
\usepackage{caption}  
\usepackage{changes}
\usepackage{amsthm}
\usepackage{textcomp}
\usepackage{amsfonts} 
\usepackage{hyperref}
\usepackage{bm}
\usepackage{epsfig}
\usepackage{wrapfig}
\usepackage{amssymb}
\usepackage{mathrsfs}
\usepackage{dsfont}
\usepackage{subfigure}
\usepackage{tikz}
\usepackage{tcolorbox}
\usepackage[framemethod=TikZ]{mdframed}
\usepackage{amsfonts,dsfont}
\usepackage{amssymb}
\usetikzlibrary{calc,patterns}
\usetikzlibrary{arrows}
\usetikzlibrary{arrows.meta}
\usetikzlibrary{decorations.markings}
\usepackage{diagbox}
\usepackage[margin=1in]{geometry}

\numberwithin{equation}{section}

\allowdisplaybreaks

\newcommand{\eqa}{\begin{eqnarray}}
\newcommand{\ena}{\end{eqnarray}}
\newcommand{\eq}{\begin{equation}}
\newcommand{\en}{\end{equation}}
\newcommand{\eqs}{\begin{eqnarray*}}
\newcommand{\ens}{\end{eqnarray*}}

\def\emptyset{\varnothing} 
\def\ti{\to\infty} 

\def\X{\mathbf{X}} 
 
\newcommand{\R}     {\mathbb{R}}

\newcommand{\N}     {\mathbb{N}} 
\renewcommand{\P}   {\mathbb{P}} 
\newcommand{\E}     {\mathbb{E}}

 \newcommand{\floor}[1]{\left\lfloor #1 \right\rfloor}

\def\1{{\mathchoice {1\mskip-4mu\mathrm l}      
{1\mskip-4mu\mathrm l} 
{1\mskip-4.5mu\mathrm l} {1\mskip-5mu\mathrm l}}} 
\newcommand{\ssup}[1] {{{\scriptscriptstyle{({#1}})}}} 
\def\comment#1{} 
 
\newcommand{\eps}{\varepsilon}

\newcommand{\Vcal}   {{\mathcal V }} 
 
\newcommand{\xx}{{\bm{x}}}
\newcommand{\zz}{{\bm{z}}}
\newcommand{\0}{{\bm{0}}}
\newcommand{\yy}{{\bm{y}}}
 
\def\ignore#1{}

\def\Def{\ :=\ }

\def\added#1{{{#1}}}

\newtheorem{thm}{Theorem}[section]

\newtheorem{definition}[thm]{Definition}

\newtheorem{cor}[thm]{Corollary}
\newtheorem{propo}[thm]{Proposition}
\newtheorem{rmk}[thm]{Remark}
\newtheorem{Ex}[thm]{Example}

\renewcommand{\epsilon}{\varepsilon}

\title[3 steps mixing  for general random walks on hypercube]{Three steps mixing  for general random walks on the hypercube at criticality.
 }
\date{}
\author[A.~Collevecchio]{Andrea Collevecchio}
\address{Andrea Collevecchio\\ School of Mathematical Sciences, Monash  University, Melbourne} \email{Andrea.Collevecchio@monash.edu}
\author[R. Griffiths]{Robert Griffiths}
\address{Robert Griffiths\\ School of Mathematical Sciences, Monash  University, Melbourne}
\email{Bob.Griffiths@monash.edu}


\begin{document}

\begin{abstract}
We  introduce a general class of random walks on the $N$-hypercube, study cut-off for the mixing time, and provide several types of representation for the transition probabilities. 
We observe that for a sub-class  of these processes  with long range (i.e. non-local) there exists a critical value of the range that allows an \lq\lq{}almost-perfect\rq\rq{} mixing in at most three steps.   In other words, the total variation distance between the three steps transition and the stationary distribution decreases geometrically in $N$, which is the dimension of the hypercube. In some cases, the walk mixes almost-perfectly in exactly two steps. Notice that a well-known  result (Theorem 1 in \cite{DS86})  shows that there exist no random walk on Abelian groups (such as the hypercube) which mixes perfectly in  exactly two steps.
\end{abstract}

\maketitle
\tableofcontents
\section{Introduction}
The field of mixing time has attracted the attention of many mathematicians in the past 30 years. It can be described as the study of the rate of  convergence of Markov chains to their stationary distribution, and has an enormous amount of applications, for example  in physics, economics, biology, combinatorics. \\
This field  of study is interesting not only for the mathematical tools developed, and their applications to real life problems, but also for the variety of different behaviour that Markov chains can exhibit. 

A {\bf cut-off phenomena} is observed in certain cases, which highlights a discontinuity in a Markov chain's behaviour, where there is a sudden change from being very far away to become very close  to stationarity.
In some cases, the chain  reaches stationarity in a finite number of steps, producing a {\bf perfect sampling} from the stationary distribution.  More frequently this happens in a random number of steps. For example consider the celebrated coupling-from-the-past technique introduced in \cite{PW1996} which had  a huge impact in  simulations of models from statistical mechanics (e.g. the Ising model). More rarely, this perfect sampling is achieved in a deterministic number of steps. 
%

In this paper, we consider a large class  $\mathcal{G}$ of reversible Markov chains $\X$,  not necessarily time homogeneous. We aim to study the behaviour of their mixing time as the state space increases, both in terms of total variation and $\chi^2$ distances. Our  results  highlight a certain discontinuity of the mixing time in terms of the size of a  single step  of the random walk. We characterize the cases when the chain mixes \lq {\bf almost- perfectly}\rq\  in  at {\bf most three steps}.  This means that the total variation distance between the distribution of the process at time 3 and the stationary distribution decreases fast to zero as  the dimension of the hypercube increases (see Definition~\ref{almost-perf}). 
For example, to  illustrate this phenomena, consider the following chain, which is described in detail in Example~\ref{Ex:2} below.  Fix $p >1/2$ and a parameter $\alpha \in (0,1]$. At each step  exactly {$\floor{\alpha N}$}  coordinates are picked uniformly at random and their value is {changed with the following procedure, which is repeated independently} for each coordinate selected.  If it is $0$ it changes to $1$, while if it is $1$ an independent randomization is used.  The 1 becomes 0 with probability {$(1-p)/p$}, and does not change otherwise.
If $\alpha = p$,  we prove that it mixes almost perfectly in 2 steps.
\begin{wrapfigure}{r}{8cm}
\begin{tikzpicture}[thick, scale=0.6]
      \draw[->] (-1,0) -- (6,0) node[right] {$\alpha$};
      \draw[->] (0,-1) -- (0,6) node[above] {$t_{\rm mix}$};
      \draw[-, thick, red] (0,5)--(2.4,5);
         \draw     (2.5,  4.98) node {o};
          \draw[-, thick, red] (2.6,5)--(5,5);
\draw (0,5)  node[left] {$\sim  \ln N$};
\draw (5,0)  node[below] {1};
\draw (2.5,-0.1)  node[below] {$p$};
\draw  [red]   (2.5,  0.3) node {\textbullet};
\draw (0,0.3)  node[left] {$2$};
\draw[dotted, thick] (0,0.3)--(2.3,0.3);
\end{tikzpicture}
\end{wrapfigure} 
The value $\alpha = p$ is what we call the critical value.  Moreover almost-perfect mixing in at most  3 steps  is observed in the window $\alpha \in [p -   v/ \sqrt{N}, p +   v/ \sqrt{N}]$ where $v$ is any real number, and can even be random.  On the other hand, if $\alpha \neq p$ the chain mixes in the order of  $\ln N$ steps, and a cut-off is proved in the $\chi^2$ distance. This unexpected discontinuity is described in Figure 1. 

 Moreover, $p$ is allowed to depend on $N$, and we find interesting the case where $p_N$ converges to $1/2$. We interpret this case as a small perturbation of the case $p=1/2$. When we compare this result  with the existing literature on long-range random walks on the hypercube with $p=1/2$, we observe a big gap, as the latter process mixes slowly, at least in the $\chi^2$ distance (see \cite{Ne2017} and  the discussion in Section~\ref{Novel} below).

The almost-perfect mixing  in exactly two steps described  above is  surprising also because of a well-known result by \cite{DS86} which implies that no random walk on an Abelian group reaches  perfect stationarity in exactly two steps.

Moreover, our result include a computable spectral representation for this class of processes, which   enables us to show the so-called  {\bf cut-off} phenomenon for a large class of processes.   \\
The class of process we consider are intimately related to the Ehrenfest Urn, and to its generalizations (see discussion in Section ~\ref{Novel} below), which in turn has direct applications to chemistry and physics (see e.g.  \cite{FPG08}).\\
We assume that each process $\X = (X_t)_t$ in the class  $\mathcal{G}$  takes values on the vertices  of a hypercube,  and its stationary distribution is the product measure of i.i.d. Bernoulli distributions with parameter $p\ge 1/2$ (see Condition 1 in Section~\ref{GEN}).  Moreover, we assume that $\X$ satisfies a \lq restriction principle\rq\ as stated in Condition 2 in Section~\ref{GEN},
{which can be roughly described as follows. The probability of any given collection of coordinates (say $B\subseteq [N]$) being updated at time $t+1$ depends on the past of the process only through the coordinates $B$ of  $X_t$, and might depend on some external randomization.}
We discuss below {a} few examples of processes satisfying the properties described above. In particular, {a} simple random walk on the hypercube , \cite{DGM1990},  and a class of non-local random walk{s} \cite{Ne2017} on the hypercube belong to $\mathcal{G}$.

Moreover, we highlight another phase transition which we find surprising. 
{The} starting position can determine   mixing in a bounded time, { in this case   for the Hamming distance},  when the update size is far from criticality.

We also provide general representations for the  {$t$}-step transition matrix  of $\X$, one in terms of a system of random walks and the other as a spectral representation. Here are {a} few examples of process that lie in the class we describe above.

\begin{Ex} {\bf Lazy simple random walk on the hypercube (RWH)}.  Let $X_0$ be a vertex of the $N$-dimensional hypercube, and define the process $(X_t)_{t\in \mathbb{N}}$ recursively as follows. Suppose that at stage $t+1$, a fair coin is flipped. If it shows Head, $X_{t+1} = X_t$. If the coin shows Tail, then  a coordinate of $X_t$   is chosen uniformly at random, and it is changed.  This process has been studied extensively.  In particular it was shown in  \cite{DGM1990} that it exhibits a cut-off at $(1/4)N \log N$.
\end{Ex} 

\begin{Ex}\label{Ex:2} {\bf Non-local random walk on the hypercube (NLRWH) 
}\ Consider the following  random walk. Fix   {parameters} $p_N \ge 1/2$ and $z_N \in [N]$. Pick a set of coordinates  with cardinality $z_N$ uniformly at random, {i.e. each possible choice is picked with  probability}
$$ {{N \choose z_{\small N}}}^{-1}.$$
For each coordinate  $i$ selected we perform the following procedure, which we call {\bf Acceptance/Rejection} with parameter $p_N$:
\begin{itemize}
\item[a)] If $X_t[i] = 0$ then $ X_{t+1}[i] =1$
\item [b)] If $X_t[i]=1$ then we randomize further, and  set
$$ X_{t+1}[i] = \begin{cases}
0 \qquad \mbox{ with probability } \frac {1-p_N}{p_N}\\
1\qquad \mbox{ otherwise }
\end{cases}.
$$
\end{itemize}
The stationary distribution is unique, and is a product  measure of i.i.d. Bernoulli\rq{}s  with parameter $p_N$. The case   $p_N\equiv 1/2$, with an additional assumption of lazyness, was studied in \cite{Ne2017}.
\end{Ex}
\begin{Ex} \label{ex:i.i.d}{\bf Mixture of i.i.d. updates for each coordinate. 
}\
 Fix $p_N \ge 1/2$. Define a {process $(X_t)_{t\in \mathbb{N}}$ recursively}. Let $X_0 = {\bf 0} \in \Vcal_N$. Suppose that we {have} defined $X_t${, then}  we obtain $X_{t+1}$ as follows. Let $I^{\ssup N}_t$ be a random variable with distribution $\nu_{N,t}$. We assume that for any fixed $N \in \N$,  the random variables $(I^{\ssup N}_t)_{t\in \mathbb{N}}$ are independent.  Given $I^{\ssup N}_t =\alpha_{N,t}$,   for each coordinate $j \in [N] = \{1, 2, \ldots N\}$  {flip} an independent coin that has probability $\alpha_{N,t}$ of showing Head. If the coin shows Tail we set $X_{t+1}[j] = X_t[j]$. The coordinate is selected if and only if the corresponding coin shows Head.  For each selected coordinate we repeat the Acceptance/Rejection procedure described in the Example \ref{Ex:2} with parameter $p_N$.\\
This process has stationary distribution product  measure of i.i.d. Bernoulli\rq{}s  with  parameter $p_N$.
\end{Ex}

\begin{Ex} {\bf Blocks update.}
Let $\beta_N$ be a sequence such that $N/\beta_N$ is a  positive integer. Partition the space $[N]$ into $N/\beta_N$ disjoint subsets with  cardinality $\beta_N$ each. Exactly one group is chosen, each with equal probability.  For each coordinate $j$ of this group we repeat the Acceptance/Rejection method described in Example \ref{Ex:2} with parameter $p$.
The Markov chain $\X$ is reversible with respect the measure  product  measure of i.i.d. Bernoulli\rq{}s  with  parameter $p$.
\end{Ex}
\section{Literature review and novelty of our results}\label{Novel}
 The main contributions of this paper can be summarized as follows.  \\
\noindent $\bullet$ {\bf Almost-perfect mixing with acceptance/rejection.}
Long range versions of RWHs have been  studied in  \cite{Ne2017}. In this context, the random walks were considered to be  \lq fair\rq, i.e. $p =1/2$, and \lq lazy\rq, i.e. at each stage the process would not change with probability 1/2. The latter assumption is convenient to avoid periodicity, and ensure ergodicity of the process. 
 {The critical case $z_N = N/2$ was discussed in Section 6 of \cite{Ne2017},} where an upper bound for the mixing time of $N$ was provided. Moreover, a lower bound for the $\chi^2$ distance was also provided, and  still of the order $N$ (see Remark 2 on page 1297 of  \cite{Ne2017}), suggesting that the chain would not mix rapidly.  This behaviour seems a bit subtle, as the mixing time for the same chain, when $z_N = \alpha N$, with $\alpha <0.5$, is of the order $\log N$. 
Our contribution, for this particular example, is to show that lazyness is the  cause of this 
slowing down in the case of  $\chi^2$ distance. If we apply the acceptance rejection method described in the examples above, with $p_N \downarrow 1/2$, and $p_N \neq 1/2$, we can observe a perfect mixing within 3 steps (see Theorem~\ref{3step} below). Notice that when $p_N  \neq 1/2$ the chain is aperiodic, as there is a positive probability for the {coordinates} not to change. Of course, this is a different model, but we can choose $p_N$ in such a way that the similarity  between the two models is quite evident. 
To see this, we can identify  the limiting distribution of $\pi_N(\cdot, 1/2)$ with a Uniform over the interval $[0,1]$. In fact, we can identify the vertices of the hypercube with a truncated binary expansion and the stationary measure is a product measure of Bernoulli$(1/2)$. In contrast, if we consider a sequence of i.i.d. $(\xi_n)_{n\in \mathbb{N}}$ of Bernoulli($p$) with $p \neq 1/2$, the limit of the $\sum_{n=1}^\infty \xi_n 2^{-n}$ has a distribution singular with respect to the Lebesgue measure.  The latter, is a consequence of a beautiful Theorem of \cite{K1948}. \\
 Hence it makes sense to consider sequences $p_N \to 1/2$. Fix $\eps>0$ and choose $p_N = 1/2 + \delta_\eps/N^a$ where $a>1$ and $\delta_\eps>0$ only depends on $a$ and $\eps$. 
Using   again Kakutani Theorem  we have that the limiting distribution of the product Bernoulli ($p_N$), using the binary expansion trick, is uniformly continuous with respect to the Lebesgue measure. Denote by $\eta_\eps(\cdot)$ this distribution. It is not difficult to prove that we can choose $\delta_\eps$ such that the total variation distance between $\eta_\eps(\cdot)$ and the uniform measure is less than $\eps$.\\
\noindent $\bullet$ {\bf  Mixing in finitely many steps at critical (deterministic) initial conditions.}
In example \ref{Ex:2}, if we choose as initial configuration a vector in the $N$-dimensional hypercube which has exactly $Np$ ones, then the mixing time of the Hamming distance of the process is  of constant order, provided $z_N = \alpha N$ for some $\alpha \in (0,1)$. In contrast, if we start with an arbitrary initial configuration, and if $\alpha \neq p$, the mixing time becomes  of the order $\ln N$ in the worst case scenario.  \\
\noindent $\bullet$ {\bf General representations.} Moreover, ours is a unifying approach which allows a study of a general class of processes. 
Our work is inspired by papers \cite{KLY1993} and \cite{DG2012}, where the Krawtchouk polynomials are used to study RWH through a spectral analysis. We    combine this approach with an acceptance/rejection method.
{
 We identify a large class of processes whose transition kernel can be decomposed using these  Krawtchouk polynomials. {This representation is explicit, in the sense that we can compute the eigenvalues,} and is used to provide sharp bounds for the mixing time in $L_1$ and $L_2$  norms (i.e. with respect the total variation and the $\chi^2$ distances, respectively).}

\section{Model and Main results}\label{GEN}
We define a class $\mathcal{G}$ of Markov chains as follows. 
Let $ {\bf X} = (X_t)_{t \in \N}$ be a reversible Markov chain with state space   $\Vcal_N = \{0,1\}^N$, for some $N \in \N$. The process  ${\bf X}$ is in the class $\mathcal{G}$ if and only if satisfies the following two conditions. \\
\noindent {\bf Condition 1} There exists a parameter $p \ge 1/2$ such that   the following  { is the unique } stationary measure for $\X$, 
\begin{equation}\label{stati} \pi_N(\yy, p)  = p^{\|\yy\|} (1-p)^{N - \|\yy\|} \Def p^{\|\yy\|} q^{N - \|\yy\|}  ,
\end{equation}
where  $\yy \in \Vcal_N$, and $\|\yy\|$ is the sum of ones appearing in $\yy$, \added{i.e. the Hamming distance between $\yy$  and $\0 = (0, 0 ,\ldots, 0) \in \Vcal_N$}. In many occasions, we drop $N$ from the notation, and simply use $\pi(\yy, p)$.\\
\noindent {\bf Condition 2} For any  $\yy \in \Vcal_N$, denote by $\yy =(y[1], y[2], \ldots y[N])$ its coordinates.  For all $B\subseteq [N]= \{1, 2, \ldots N\}$ let $\yy(B)$ be the  projection from $\Vcal_N$ on $B$ defined as  the vector $\yy(B) = (\yy[j], j \in B)$. We assume that
\begin{equation}
\mathbb{P}(X_t(B)\in C\mid X_{t-1}) = \mathbb{P}(X_t(B)\in C\mid X_{t-1}(B)).
\label{subsets:0}
\end{equation}
For any pair of probability measures $\mu$ and $\nu$ defined on a countable  space $\Omega$, define the total variation distance
$$
\|\mu - \nu\|_{TV} = \max_{A \subseteq \Omega} |\mu(A) - \nu(A)|. 
$$
\begin{definition}\label{defZ}
Let $\X \in \mathcal{G}$. 
Define the sequence    $(Z_t)_{t \in \N}$  of independent random vectors in $\Vcal_N$, with the following  distribution
\begin{equation}
\mathbb{P}(Z_t \in S) = \mathbb{P}(X_t \in S\mid  X_{t-1} =\bm{0}).
\label{Z0:0}
\end{equation}
 From now on, we denote the coordinates of $Z_t$ by $(Z_t[1], Z_t[2], \ldots Z_t[N])$. 
\end{definition}
\begin{rmk}
In what follows, we denote by $ P_t(\cdot\;|\; \xx)$ the probability mass function of $X_t$ given $X_0 = \xx$. Moreover, when we consider a generic $\X \in \mathcal{G}$ we denote by $N$ the dimension of the corresponding hypercube. 
\end{rmk}
We have the following representation.
\begin{tcolorbox}

\begin{thm}[{\bf Spectral Representation}] \label{thm2} Let $\X \in \mathcal{G}$. We have,
\begin{equation}\label{trans:20}
P_t(\yy\mid \bm{x}) = \pi(\yy, p) \Bigg \{1 + \sum_{A\subseteq [N], A \ne \emptyset}\left(\prod_{m=1}^t \rho_{A, m}\right) \left(\frac{p}{q}\right)^{|A|}\prod_{j\in A}
{\Big (1-\frac{\xx[j]}{p}\Big )\Big (1-\frac{\yy[j]}{p}\Big )\Bigg \}.}
\end{equation}
{where we can give an explicit representation for the eigenvalues, i.e.}
\begin{equation}
\rho_{A, m} = \mathbb{E}\left [\prod_{j\in A}\Big ( 1 - \frac{Z_m[j]}{p}\Big )\right ].
\label{rhoZ:0}
\end{equation}
\end{thm}
\end{tcolorbox}
\begin{rmk}
The spectral representation in \eqref{trans:20} simplifies when $\X$ is time-homogenous and instead of a product, we simply have $\rho_A^t$. Notice that the previous representation holds also in cases when the chain is reducible and/or periodic. 
\end{rmk}

If $\rho_{A,m}$ only depends on $|A|$ then we denote $\rho_{|A|} = \rho_{A,m}$.
\begin{Ex}
Let $q=p=1/2$ and take $N$ take as even. The elements of $Z_m$ are taken to be exchangeable and $||Z_m||=N/2$ with probability $1$. 
Each term in the product expression for (\ref{rhoZ:0}) is either $1$ or $-1$ according to whether $Z_m[j]$ is 0 or 1. 
There is a hypergeometric probability of $k$ terms in the product appearing in the right-hand side of \eqref{rhoZ:0} being minus one.  If $|A|=n$
\[
\rho_n = \sum_{k=0}^n 
\frac{{N/2\choose k}{N/2\choose n-k}}{{N\choose n}}(-1)^k.
\]
Simplification shows that $\rho_n=0$ if $n$ is odd, and for $m\leq N/2$,
\[
\rho_{2m} = (-1)^m
\frac{
{N/2\choose m}
}
{
{N\choose 2m}
}.
\]
The maximum value of $|\rho_n|$ is $1$, when $m=N/2$. {These eigenvalues have  appeared, e.g., in \cite{Ne2017}.}
\end{Ex}
\begin{rmk}\label{abcd:0}{\bf De Finetti sequences.}
In example~\ref{ex:i.i.d}, {when} $\X$  is time homogenous, and $\nu_{N, t}$ is a Dirac mass at $r$,   we have 
\[
\rho_{n,t}= \rho_n = \sum_{k=0}^n{n\choose k}r^k(1-r)^{n-k}\Big (-\frac{q}{p}\Big )^k = \Big ( 1 - \frac{r}{p}\Big )^n
\]
If $r=p$ then $\rho_n=0$ for $n=1,\ldots {N}$. {$\X$ is then an independence chain which mixes in one step.} If $r\ne p$ then $|\rho_1|$ is the maximum value of $|\rho_n|$.
More generally, when $\nu_{N, t} \equiv \nu_N$, we have 
\begin{equation}\label{eq:eigdf}
\rho_{n,t} = \rho_n =\int_{[0,1]} \Big ( 1 - \frac{r}{p}\Big )^n\nu_N(dr).
\end{equation}
For each fixed $N$, the coordinates of $Z$ can represent $N$ particular coordinates from a countably infinite de Finetti sequence with mixing measure $\nu_N$.  { If $\nu_N \equiv Leb(0,1)$ then 
\[
\rho_{n,t} = \rho_n =\int_{[0,1]} \Big ( 1 - \frac{r}{p}\Big )^n d r= \frac p{n+1} \left( 1 - \Big(-\frac qp\Big)^{n+1}\right).
\]
}
\end{rmk}

\begin{tcolorbox}
\begin{definition} We define the collection $\mathcal{C}$ of sequences $(\X^{\ssup N})_{N\in \mathbb{N}} \in \mathcal{G}^{\N}$ with the following property. For each $N \in \N$, 
\begin{itemize}
\item the state space of  $\X^{\ssup N}$ is $\Vcal_N$, and 
\item there exists a sequence $(p_N)_N \in [1/2, 1]^{\N}$ such that $\pi_N(\cdot, p_N)$ is a stationary distribution for $\X^{\ssup N}$, and $\lim_N p_N = p$ for some $p \in [1/2, 1]$.
\end{itemize}
From now on, once an element of $\mathcal{C}$ is fixed, we denote by $P_t^{\ssup N}$ the transition kernel  of $\X^{\ssup N}$. 
\end{definition}
\end{tcolorbox}


Define 
\begin{equation}\label{tmix}
t_{mix}(\eps, \xx)  = \inf\{t \colon \|P^{\ssup N}_t(\cdot\;|\; \xx)- \pi_N(\cdot)\|_{TV} \le \eps\}.\\
\end{equation}
Let $t_{mix}(\eps) = \sup_{\xx \in \Vcal_N} t_{mix}(\eps, \xx)$.\\

Theorem ~\ref{th-low-t} is quite general and simple to prove. This result provides almost the correct order for the mixing time, missing a logarithmic factor.
It provides bounds that are sharp up to a logarithmic factor in the  case of exchangeability (defined in \eqref{eq:exch1}  below). 
\begin{tcolorbox}
\begin{thm} [{\bf General lower bound for $\mathbf{t_{mix}}$}]\label{th-low-t} Suppose that $(\X^{\ssup N})_N \in \mathcal{C}$  and each process in the sequence is time homogeneous, i.e. $(Z^{\ssup N}_i)_{i\in \mathbb{N}}$ are identically distributed for each $N$. Define 
\begin{equation}\label{def:theta}
\theta_N = \min_{j \in [N]} \P(Z^{\ssup N} [j] = 1).
\end{equation}
There exists $a>0$ such that   $t_{mix}(\eps) \ge a \theta_N^{-1}$, where we set $a/0 = \infty$. Notice that $\theta_N >0$ guarantees irreducibility of the Markov chain $\X^{\ssup N}$ .
\end{thm}
\end{tcolorbox}

\begin{definition}\label{def-ex}
A random variable $Z$ which takes values on $\Vcal_N$ is said to be exchangeable if 
\begin{equation}\label{eq:exch1} \P(Z= \xx) = \P(Z = \yy)
 \qquad \mbox{whenever $\|\xx\| = \|\yy\|$.}
 \end{equation}
\end{definition}

\begin{definition} \added{Let $\P$ be a   measure and $\mathbb{Q}$ a positive measure  both defined on the subsets of $\Vcal_N$. Define 
\begin{equation}\label{eq:defchi}
\chi^2 (\P\;|\mathbb{Q})  = \sum_{\yy \in \Vcal_N}\frac{\big (\P(\yy) - \mathbb{Q}(\yy)\big )^2}{\mathbb{Q}(\yy)}.
\end{equation}
We set $\chi^2(\xx, t) = \chi^2 (P_t(\cdot\;|\; \xx) \;|\pi_N) $.}
Moreover, let 
$$t^{\ssup 2}_{mix} (\eps, \xx)= \inf\{t \colon \chi^2(\xx, t) \le \eps\},$$
and \added{$t^{\ssup 2}_{mix}(\eps) = \sup_{\xx \in \Vcal_N} t_{mix}(\eps, \xx).$}
\end{definition}

\begin{definition}\label{almost-perf} Consider  a sequence $(\X^{\ssup N})_N \in \mathcal{C}$.  We say that this sequence mixes {\bf almost perfectly} in $t_0$ steps, if there exists  constants $C$ and $\beta \in (0,1)$ such that  
\begin{equation}
 \sup_{\xx \in \Vcal_N}\|P_{t_0}(\cdot\;|\; \xx) - \pi_N(\cdot, p_N)\|_{TV} \le C \beta^N.
\end{equation}
\end{definition}

\begin{tcolorbox}
\begin{thm}[{\bf Almost-perfect mixing in three steps}]\label{3step}
Consider a sequence $(\X^{\ssup N})_N \in \mathcal{C}$.  We make the following assumptions.
\begin{itemize}
\item  $p_N >1/2$ for all $ N \in \N$ and each $\X^{\ssup N}$ is time-homogenous. 
\item Each of the random variables $Z^{\ssup N}_1, Z_2^{\ssup N}, Z_3^{\ssup N}$ is exchangeable,  for each $N \in \N$, in the sense of definition~\ref{def-ex}. 
\item Let $\zeta_N = \|Z^{\ssup N}_1\|$.   We assume that there exists a random variable $V$ such that 
\begin{eqnarray}
&\lim_{N \ti} \frac{\zeta_N  - Np}{ \sqrt{Npq}} =  V \qquad \mbox{(in distribution)}\\
&\sup_N\E\left[\frac{(\zeta_N  - Np)^a}{ (Npq)^{a/2}}\right]<\infty \qquad \mbox{for some $a >1$,}\label{eq:mom}\\
&\E[{\rm e}^{V^2/2}] < \infty,
\end{eqnarray}
\end{itemize}
Then, we have that $(\X^{\ssup N})_{N\in \mathbb{N}}$ mixes almost-perfectly in 3 steps.  \\
Moreover, if $\P(V = 0) =1$ then  $(\X^{\ssup N})_{N\in \mathbb{N}}$ mixes almost-perfectly in 2 steps.
\end{thm}
\end{tcolorbox}

\begin{tcolorbox}
\begin{thm} [{\bf Cut-off for NLRWH}]\label{thm5} Let $(\X^{\ssup N}) \in \mathcal{C}$ and assume that each $\X^{\ssup N}$ are defined as  in Example~\ref{Ex:2}, with stationary distribution $\pi_N(\cdot, p_N)$.
\begin{enumerate}
 \item  If $\lim z_N/N = 0$, then both  ${t}_{mix}(\eps)$ and  ${t}^{\ssup 2}_{mix} (\eps)$ {exhibt} a sharp cutoff at  $\frac{Np_N}{2z_N}\log N$. In other words, if we set 
$$ t_C = \frac{Np_N}{2z_N}(\log N + C),$$
we have that for all $C<0$ small enough, ${t}_{mix}(\eps), {t}^{\ssup 2}_{mix} (\eps)>t_C$, and for all  $C$  large enough ${t}_{mix}(\eps), {t}^{\ssup 2}_{mix} (\eps)<t_C$.
\item {If $z_N/N  = w  \in (0,1]\setminus{\{p\}}$,  then  ${t}^{\ssup 2}_{mix} (\eps)$ exhibits a sharp cutoff at  $t_C = (p_N/(2 w))(\log N +C)$. }
\end{enumerate}
\end{thm}
\end{tcolorbox}
 \added{Let $\overline{P}_t(\cdot\;| \xx)$ be the p.m.f. of $\|X_t\|$ conditional on $X_0 = \xx$. Let $\mathbb{Q}_N$ be a Binomial with parameters $N$ and $p$. Define $\chi^2$ for the Hamming distance as $\chi_H^2(\bm{x}, t) =\chi^2(\overline{P}_t(\cdot\;| \xx)\;|\; \mathbb{Q}_N)$.
 $$\overline{t}^{\ssup 2}_{mix} (\eps) = \inf\{t \colon \max_{\xx} \chi_H^2(\bm{x}, t) \le \eps\}.$$
 }
 \begin{rmk}\label{abcd:1.0}{\bf De Finetti sequences-continued.}
Suppose that in the De Finetti case described in Remark~\ref{abcd:0} we set $\nu_{N,t} \equiv Leb(0,1)$ and $p>q$. Let $t = a N/(\ln N)$ where $a >0$ is a small enough constant to be specified below. Using the computation given in Remark~\ref{abcd:0} we have that 
\begin{equation}\label{eq:compdf1}
\begin{aligned}
\chi^2_H (\0, t)&=\sum_{n=1}^N {N \choose n} \left(\frac pq \right)^n\left(\frac p{n+1}\right)^{2t} \left(1- \Big(- \frac qp \Big)^{n+1}\right)^{2t} \\
&\ge q^{-N} N^{-2t} {N \choose \floor{pN}} p^{\floor{pN}} q^{N-\floor{pN}} \left(\frac{Np}{\floor{pN}+1} \right)^{2t}  \left(1- \Big(- \frac qp \Big)^{\floor{pN}+1}\right)^{2t}\\
&=q^{-N} N^{-2t}  \frac{1}{\sqrt{2 \pi  pq N}} (1 + o(1))\\
&= \exp\{ - N \ln q  - 2 a N \} \frac{1}{\sqrt{2 \pi  pq N}}(1 + o(1)).
\end{aligned}
 \end{equation}
 Hence, if $a < -(\ln q)/2$, we have that $\lim_{t \ti} \chi^2_H (\0, t) = \infty$. Hence $\overline{t}^{\ssup 2}_{mix} (\eps) > a N/\ln N$.This result should be compared with the case were $\nu_N$ is a dirac mass at a point $\alpha \in (0,1)\setminus\{p\}$, i.e. i.i.d. updates.  In this context,
 \begin{equation}\label{eq:i.i.d}
\begin{aligned}
 \chi^2_H (\0, t)&=\sum_{n=1}^N {N \choose n} \left(\frac pq \right)^n \left(1 - \frac{\alpha}{p}\right )^{2tn}  = \left( 1 + \frac pq\Big(1 - \frac{\alpha}{p} \Big)^{2t}  \right)^N -1.
 \end{aligned}
 \end{equation}
 The latter equation shows a completely different behaviour. In fact, we show in Section~\ref{se:digr} that  $\chi^2_H (\0, t) = \sup_{\xx \in \Vcal_N} \chi^2_H (\xx, t)$. Equation \eqref{eq:i.i.d} shows that when $\nu_{N, t}$ is a dirac mass  at $\alpha \neq p$ then it has a cut-off at $b \ln N$ when we consider the $\chi^2$ distance, with $b$ depending on $\alpha$ only. Hence, it mixes much faster than the case when $\nu_{N,t} \equiv Leb(0,1)$, which requires at least $a N/\ln N$ steps to mix. The i.i.d. case  that we just discussed was studied  in \cite{Sc2011}.
\end{rmk}
 \vspace{-.7cm}
 \begin{tcolorbox}
\begin{thm}[{\bf Constant order mixing at critical initial conditions}]\label{thm6}
 \added{Fix $\eps >0$.  Let $(\X^{\ssup N}) \in \mathcal{C}$ and assume that each $\X^{\ssup N}$ \added{is} defined as  in Example~\ref{Ex:2}, with stationary distribution $\pi_N(\cdot, p_N)$. If $\|\xx\|/N=p$  and \added{$z_N = w N$, with $w >0$}, then there exists $t_\eps$ not depending on $N$ such that   $\overline{t}^{\ssup 2}_{mix} (\eps) \le  t_\eps$.}
 \end{thm}
\end{tcolorbox}
The following representation characterises the process in $\mathcal{G}$ in terms of $N$ (possibly) dependent random walks. 
\begin{tcolorbox}
\begin{thm}[{\bf Random Walk Representation}]\label{thm1} Suppose that  $\bf{X} \in \mathcal{G}$. We have
\begin{equation}\label{eq:rwr}
\begin{aligned}
P_t(\yy\mid \bm{x}) = \pi(\yy, p)\mathbb{E}\Bigg [\prod_{j=1}^N\Bigg (1 + \Big (-\frac{q}{p}\Big )^{S^{\ssup t}[\xx, \yy, j]}\Bigg )\Bigg]
\end{aligned}
\end{equation}
where $\bm{S}^{\ssup t}[\xx, \yy, j]= \xx[j] + \yy[j] -1 + \sum_{k=1}^t Z_k[j],$ and the parameters $p \ge q$ are the same as in \eqref{stati}. The vectors  $(Z_j)_{j\in \mathbb{N}}$ are independent and their distribution is defined in Definition~\ref{defZ}. Moreover,
\begin{equation*}
\begin{aligned}
\hspace{-0.5cm}\|P_t(\cdot\mid \bm{x}) - \pi(\cdot)\|_{TV} =2 \sum_{\yy \in \Vcal_N} \;\pi(\yy, p) \left|\mathbb{E}\Bigg [\prod_{j=1}^N\Bigg (1 + \Big (-\frac{q}{p}\Big )^{S^{\ssup t}[\xx, \yy, j]}\Bigg )\Bigg] - 1 \right|.
\end{aligned}
\end{equation*}
Viceversa, if the random vectors $(Z_j)_{j\in \mathbb{N}}$ are independent, then the process with transition functions defined as in \eqref{eq:rwr} belongs to $\mathcal{G}$.
\end{thm}
\end{tcolorbox}
\section{Spectral representation via tensor products}
\begin{propo} If $\X \in \mathcal{G}$  then  there exists constants $(\gamma_{A, t})_{A \subseteq [N], t \in \N}$ such that 
\begin{equation}\label{eq:onestep}
P_t(\yy\mid \bm{x}) = \pi(\yy, p)\Bigg \{1 + \sum_{A\subseteq [N], A \ne \emptyset}\left(\prod_{m=1}^t \gamma_{A, m}\right) \left(\frac{p}{q}\right)^{|A|}\prod_{j\in A}\left(1-\frac{\xx[j]}{p}\right)\left(1-\frac{\yy[j]}{p}\right)\Bigg \}.
\end{equation}
\end{propo}
\begin{proof}
The general form of a 1-step  transition density expansion for $\X$ is
\begin{equation}
\P(X_{t+1} = \yy\;|\;X_{t} = \xx) = \pi(\yy, p)\Bigg \{1 + \sum_{L, M\subseteq [N], L, M \neq \emptyset}\gamma_{LM}^{\ssup t}\prod_{i\in L}\frac{p-\xx[i]}{\sqrt{pq}}\prod_{j\in M}\frac{p-\yy[j]}{\sqrt{pq}}\Bigg \},
\label{tensor:20}
\end{equation}
with $\gamma^{\ssup t}_{LM}=\gamma^{\ssup t}_{ML}$.
This is a  well-known expansion, named after   \cite{L1969}, for $P_t(\yy\mid \bm{\xx})/\pi(\xx)$ (also known as Fourier-Walsh basis expansion in part of the literature) using the tensor product sets
\begin{equation}\label{eq:ortho}
\Bigg \{\bigotimes_{i=1}^N\Big \{1,\frac{p-\xx[i]}{\sqrt{pq}}\Big \}\Bigg \}~\bigotimes ~\Bigg \{\bigotimes_{j=1}^N\Big \{1,\frac{p-\yy[j]}{\sqrt{pq}}\Big \}\Bigg \}.
\end{equation}
The following steps are well-known from basis theory, but we include the steps for the sake of completeness. We emphasize that we can compute the eigenvalues  explicitely. Roughly speaking, the  \emph{Lancaster} expansion applies to the ratio $P_t/\pi$ in terms of the two tensor product sets which are complete orthogonal function sets on the Bernoulli product distributions on the sequences. 
The symmetry $\gamma^{\ssup t}_{LM}=\gamma^{\ssup t}_{ML}$ is  a consequence of  reversibility of $\X$. 
 Moreover, for $L \nsubseteq M$, where the p.m.f. of $(X_{t+1}, X_t)$ is 
 $
 P_{t+1}(\bm{x}_{t+1}\mid \bm{x}_t)\pi (\bm{x}_t,p)
 $
 ,we have
\begin{eqnarray*}
\gamma^{\ssup t}_{LM} &=& \mathbb{E}\Big [\prod_{i\in L}\frac{p-X_t[i]}{\sqrt{pq}}\prod_{j\in M}\frac{p-X_{t+1}[j]}{\sqrt{pq}} \Big ]
\nonumber \\
&=&\mathbb{E}\Big [\prod_{i\in L}\frac{p-X_t[i]}{\sqrt{pq}}\mathbb{E}\Big [\prod_{j\in M}\frac{p-X_{t+1}[j]}{\sqrt{pq}}\mid X_t\Big]\Big]
\nonumber \\
&=&\mathbb{E}\Big [\prod_{i\in L\backslash M}\frac{p-X_t[i]}{\sqrt{pq}}\mathbb{E}\Big [\prod_{j\in M}\frac{p-X_{t+1}[j]}{\sqrt{pq}}\prod_{k\in L\cap\ M}\frac{p-X_t[k]}{\sqrt{pq}}\mid X_t(M)\Big]\Big] = 0
\end{eqnarray*}
if $L \nsubseteq M$, in virtue of Condition 2.   Using \eqref{eq:onestep} we get the general representation for $P_t(\yy\;|\; \xx)$ in \eqref{trans:20}, using the orthogonality of the functions \eqref{eq:ortho}.
\end{proof}
Also the reverse is true.
\begin{propo}  If $\X$ is a reversible Markov process which satisfies condition 1 and  whose transition kernel satisfies \eqref{eq:onestep}, then $\X \in \mathcal{G}$.
\end{propo}
\begin{proof}
It is enough to prove that $\X$ satisfies Condition 2. The marginal distribution of $X_{t+1}(B)$ conditional to the event $X_t=\xx$, is
\begin{equation}\label{restr1}
\begin{aligned}
\P(&X_{t+1}(B) = \yy(B)\;|\; X_{t}(B) = \xx) =\\
&\pi(\yy(B), p)\Bigg \{1 + \sum_{A\subseteq B, A \ne \emptyset}\rho_{A, t}\Big (\frac{p}{q}\Big )^{|A|}\prod_{j\in A}\left(1-\frac{\xx[j]}{p}\right)\left(1-\frac{\yy[j]}{p}\right)\Bigg \},
\end{aligned}
\end{equation}
where $\pi(\yy(B), p)$ is the product Bernoulli$(p)$ distribution on the coordinates $B$.
The right-hand side of \eqref{restr1} depends on $\xx(B)$ only, and this proves our result.
\end{proof}
\section{Proof of Theorem~\ref{thm2}}
\subsection{General construction of the process $\X$}
In this Section we provide {a} construction for any reversible  $\X$ on $\Vcal_N$ which satisfies Conditions 1 and 2.  \added{We  explicitly construct a collection of reversible Markov processes $\mathcal{G}\rq{}$, using an acceptance/rejection method. Soon after, we prove that $\mathcal{G}= \mathcal{G}\rq{}$ (Theorem~\ref{constr} below). }\\
 A process $\X \in \mathcal{G}\rq{}$ if and only if it can be constructed  as follows.
Let $q, p$ as in \eqref{stati}, and recall $q \le p$, and $q+p =1$.
 Consider a sequence of independent random variables $({Z}_t)_{t \in \N}$ which take values in $\Vcal_N$, and let $(\xi_{i,t})_{t \in \N, i \in [N]}$ be a  sequence of i.i.d. Bernoulli($q/p$), i.e. $\P(\xi_{i,t} =1)  =q/p = 1- \P(\xi_{i,t} =0)$.
 Consider the following homogeneous Markov process, $ {\bf {X}} = ({X}_t)_{t \in \N}$, which we define recursively. Suppose  $({X}_i\colon i \le t)$ is defined, then define ${X}_{t+1}$ as follows. For all $i\in [N]$, 
\begin{itemize}
\item If ${Z}_t[i] =0$ then ${X}_{t+1}[i] = {X}_{t}[i]$.
\item If ${X}_t[i] =0$ and ${Z}_t[i] =1$ then ${X}_{t+1}[i] =1.$
\item If ${X}_t[i] =1$ and ${Z}_t[i] =1$, then ${X}_{t+1}[i] = {X}_t[i] +\xi_{i,t}\;$ mod $2$.
\end{itemize}

\begin{thm}\label{constr} $\mathcal{G} = \mathcal{G}\rq{}$.
\end{thm}
\begin{proof}
\added{We  first prove  that $\mathcal{G}\rq{} \subseteq \mathcal{G}$.} More specifically, we prove that 
\begin{equation}\label{transit-tilde} 
\hspace{1.6cm} \P({X}_{t+1} = \yy\;|\; {X}_{t} = \xx) =  \pi(\yy, p)\Big \{1 + \sum_{A\subseteq [N], A \neq \emptyset}\rho_{A, t}\Big (\frac{p}{q}\Big )^{|A|}\prod_{j\in A}\Big (1-\frac{\xx[j]}{p}\Big )\Big (1-\frac{\yy[j]}{p}\Big )\Big \}.
\end{equation}
with
\begin{equation}\label{eq:r}
 \rho_{A, t} = \mathbb{E}\Big [\prod_{i\in A}\left(1-\frac{Z_1[i]}p\right)\Big].
 \end{equation}
  Using orthogonality, it is enough to  prove the case $t =0$. Assume $\X \in \mathcal{G}\rq{}$.  A coordinate $i$ is \lq picked\rq\ if and only if $Z_1[i] =1$. 
Recall that conditionally on $Z_1$, the  coordinates that are  picked behave  independently. Hence,
\begin{eqnarray*}
&&\mathbb{E}\big [{X}_1[i]-p\mid {X}_0 = \xx, Z_1\big]  =(1-Z_1[i])(\xx[i]-p)\\
&&~+Z_1[i]\big ((1-{X}_1[i])(1-p)+X_1[i](-p(q/p) + (1-p)(1-(q/p))\big )
\nonumber \\
&&= \left(1-\frac{Z_1[i]}p\right)(X_1[i]-p).
\end{eqnarray*}
Therefore for $A \subseteq [N]$
\[
\mathbb{E}\Big [\prod_{j\in A}({X}_1[j]-p)\mid X_0 = \xx \Big ] = \mathbb{E}\Big [\prod_{i\in A}\left(1-\frac{Z_1[i]}p\right)\Big ] \prod_{j\in A}{(\xx[j]-p)}
\]
which implies \eqref{transit-tilde}, and in particular identifies 
$$ \rho_{A, 1}  =  \mathbb{E}\Big [\prod_{i\in A}\left(1-\frac{Z_1[i]}p\right)\Big].$$

 because the coefficients of 
$$\prod_{j\in A}\Big (1 -\frac{\yy[j]}{p}\Big )$$ 
in a tensor product expansion of 
$\P(X_{t+1}=\yy\mid X_t =\xx)/ \pi(\yy, p)$ with respect to $\bm{y}$ are equal to 
\[
	\rho_{A,t}\Big (\frac{p}{q}\Big )^{|A|}\prod_{j\in A}\Big (1-\frac{\xx[j]}{p}\Big ).
\]
\added{Next we  prove  that $\mathcal{G}\subseteq \mathcal{G}\rq{} $.} Take $Z_t$ to have the distribution of $X_{t+1}\mid X_t = \bm{0}$. $Z_t$ has a p.m.f.
\begin{equation*}
\pi(\zz , p)\Bigg \{1 + \sum_{A\subseteq [N], A \ne \emptyset}\rho_{A, t}\Big (\frac{p}{q}\Big )^{|A|}\prod_{j\in A}
\Big (1-\frac{\zz[i]}{p}\Big )\Bigg \}
\end{equation*}
and it follows that for $A \subseteq [N]$
\[
\mathbb{E}\Big [\prod_{j\in A}\Big ( 1 - \frac{Z_t[j]}{p}\Big )\Big ] = \rho_{A,t}.
\]
As $(\rho_{A,t})_{A, t}$ identify the distribution of the reversible markov chain $\X$, the proof follows from the following considerations. Construct a process $\X\rq{}$ in $\mathcal{G}\rq{}$
using the same $(Z_t)_{t \in \N}$. Denote by $(\rho_{A,t}\rq{})_{A,t}$ the  eigenvalues of $\X\rq{}$, then by the previous part of this proof ($\mathcal{G}\rq{} \subseteq \mathcal{G}$) we obtain that $(\rho_{A,t}\rq{})_{A,t} = (\rho_{A,t})_{A,t}$.

\end{proof}




\section{Proof of Theorem~\ref{thm1}: Random Walk representation}

\begin{proof}[Proof of Theorem~\ref{thm1}]
The transition probability for $X_1$ given $X_0=\bm{x}$ is
\begin{equation}
\pi(\yy, p)\Bigg \{1 + \sum_{A\subseteq [N], A \ne \emptyset}\rho_A\Big (\frac{p}{q}\Big )^{|A|}\prod_{j\in A}
\Big (1-\frac{\yy[i]}{p}\Big )\Big (1-\frac{\xx[i]}{p}\Big )\Bigg \}.
\label{gen:0}
\end{equation}
Note the identity that for $u\in \{0,1\}^N$, $A\subset [N]$,
\[
\prod_{j\in A}\Big ( 1 - \frac{u[j]}{p}\Big ) = \Big (-\frac{q}{p}\Big )^{||u(A)||}.
\]
%
The expression in (\ref{gen:0}) can therefore be written as
\begin{eqnarray}
&&\pi(\yy, p)\Bigg \{1 + \sum_{A\subseteq [N], A \ne \emptyset}(-1)^{|A|}
\mathbb{E}\Big [ \Big (-\frac{q}{p}\Big )^{\|Z_1(A)\|+\|\bm{y}(A)\|+\|\bm{x}(A)\|-|A|}\Big ]\Bigg \}
\nonumber \\
&=&\pi(\yy, p)\mathbb{E}\Bigg [\prod_{j=1}^N\Bigg (1 - \Big (-\frac{q}{p}\Big )^{Z_1[j]+\bm{x}[j]+\bm{y}[j] - 1}\Bigg )\Bigg ].
\label{gen:3}
\end{eqnarray}
Eq. (\ref{eq:rwr}) follows by replacing $\rho_A$ in (\ref{gen:0}) with
\[
\rho_A^t = \mathbb{E}\Big [\Big (-\frac{q}{p}\Big )^{\sum_{k=1}^t\|Z_k|(A)\|}\Big ]
 = \mathbb{E}\Big [\Big (-\frac{q}{p}\Big )^{S^t(A)}\Big ],
\]
\added{where we used the  i.i.d. assumption on the  sequence of vectors $(Z_t)_{t\in \mathbb{N}}$.}
\end{proof}
\added{Next we introduce a family of orthogonal polynomials on the Binomial distribution. }
\begin{tcolorbox}
\begin{definition}[{\rm \bf Krawtchouk~polynomials}]
We define  a class  of polynomials $\Big \{Q_n(x;N,p)\colon n, N \in \N, x \in \{0,1,\ldots,N\} \Big \}$, using the generating function
\begin{equation}
\sum_{n=0}^N{N\choose n}Q_n(x;N,p)s^n = (1-(q/p)s)^x(1+s)^{N-x}.
\label{genfn:0}
\end{equation}
\end{definition}
\end{tcolorbox}
\added{
\begin{propo}\label{pro:po}
The family of polynomials $\Big \{Q_n(x;N,p)\colon n, N \in \N, x \in [N], \Big \}$ satisfy the following properties. 
\begin{enumerate}
\item They are orthogonal in the following sense:
$
\mathbb{E}\big [Q_n(X;N,p)Q_m(X;N,p)\big ] = \delta_{m,n}h_n^{-1},
$
where $X$ is Binomial $(N,p)$, $h_n= {N\choose n}(p/q)^n$ and the Kronecker $\delta_{m, n} \in \{0,1\}$ equals $1$ if and only if $m=n$.
\item If $\xx \in \Vcal_N$  then  the family of polynomials satify a symmetric function representation
\begin{equation}
Q_n(\|\xx\|;N,p) = {N\choose n}^{-1}\sum_{A\subseteq [N],|A|=n}\prod_{j\in A}\Big ( 1 - \frac{\xx[j]}{p}\Big ).
\label{symmetric:100}
\end{equation}
\end{enumerate}
\end{propo}
The representation (\ref{symmetric:100}) is seen by noting that the generating function agrees with (\ref{genfn:0}), since
\begin{eqnarray*}
1+\sum_{n=1}^Ns^n
\sum_{A\subseteq [N],|A|=n}\prod_{j\in A}\Big ( 1 - \frac{\xx[j]}{p}\Big )
&=&
\prod_{j=0}^N\Bigg (1 + s\Big ( 1 - \frac{\xx[j]}{p}\Big )\Bigg )
\nonumber \\
&=&
\Big ( 1 - (q/p)s\Big )^{||\xx||}(1+s)^{N-||\xx||}.
\end{eqnarray*}
These polynomials are scaled so that for all $n \in [N]$, $Q_n(0;N,p) = 1$.  The relationship with the \lq usual\rq\ Krawtchouk polynomials $K_n$ is that for any $x \in \N$,
\[
Q_n(x;N,p) = \frac{K_n(x;N,p)}{ \frac{N!}{(N-n)!}(-p)^n}.
\]
See, eg. \cite{OLBC2010} Section 18.9, or \cite{DG2012} for more details about the Krawtchouk~polynomials.}
\begin{propo} If  $(Z_t)_{t\in \mathbb{N}}$ are exchangeable we have
\[
\rho_A \equiv \rho_{|A|} = \mathbb{E}\big [Q_{|A|}(\|Z_1\|;N,p)\big ].
\]
The transition probabilities are then
\begin{equation}
\pi(\yy, p)\Bigg \{1 + \sum_{n=1}^N \rho_n\Big (\frac{p}{q}\Big )^{n}\sum_{A\subseteq [N], |A|=n}
\prod_{j\in A}\Big (1-\frac{\xx[j]}{p}\Big )\Big (1-\frac{\yy[j]}{p}\Big )\Bigg \}.
\label{gen:10}
\end{equation}
\end{propo}
\begin{proof}
%
This follows from (\ref{symmetric:100}) since under exchangeability for any $A\subset [N]$ with $|A|=n$,
\[
\rho_A=\mathbb{E}\Big [\prod_{j\in A}\Big (1 - \frac{Z_1[j]}{p}\Big )\Big ]
= {N\choose n}^{-1}\sum_{A\subseteq [N],|A|=n}\prod_{j\in A}\mathbb{E}\Big [\Big ( 1 - \frac{Z_1[j]}{p}\Big )\Big ]=\mathbb{E}\big [Q_{|A|}(\|Z_1\|;N,p)\big ].
\]
\end{proof}

\begin{cor}
The transition probabilities (\ref{gen:10}) can be written as 
\begin{equation}
\pi(\yy, p)\Bigg \{1 + \sum_{n=1}^N \rho_nh_n
R_n\big (\|\bm{x}\|,\|\bm{y}\|,  \langle\bm{x},\bm{y} \rangle\big )\Bigg \},
\label{gen:101}
\end{equation}
where $\langle \cdot \rangle$ denotes inner product and $R_n(\cdot, \cdot, \cdot)$ is the coefficient of ${N\choose n}s^n$ in the generating function
\begin{equation}\label{defR}
(1+s)^{N_{00}}(1 - sq/p)^{N_{01}+N_{10}}(1+sq^2/p^2)^{N_{11}}
\end{equation}
with $N_{lm}$ being the number of pairs $\bm{x}[j]=l,\bm{y}[j]=m$, $j \in [N]$, $l,m\in \{0,1\}$. The counts appearing in \eqref{defR} satisfy
\begin{equation*}
N_{00}= 1 - ||\bm{x}||-||\bm{y}||+ \langle\bm{x},\bm{y} \rangle,N_{01}+N_{10}= \|\bm{x}\|+\|\bm{y}\|-2  \langle\bm{x},\bm{y} \rangle,N_{11}=  \langle\bm{x},\bm{y} \rangle.
\end{equation*}
\end{cor}

\begin{propo}
Suppose $Z_1$ is exchangeable. Fix $\yy, \xx \in \Vcal_N$. We have 
\begin{equation}
\P(\|X_1\| = \|\yy\|\;|\;X_0 = \|\xx\|) =  {N\choose \|\yy\|}p^{ \|\yy\|}q^{N- \|\yy\|}\Big \{1 + \sum_{n=1}^N\rho_nh_nQ_n(||\bm{x}||;N,p)Q_n(||\bm{y}||;N,p)\Big \}.
\label{H:10}
\end{equation}
\end{propo}
\begin{proof}
The sum in a permutation distribution $\sigma$ in the symmetric group of order $N$ 
\begin{eqnarray*}
&&\frac{1}{N!}\sum_{\sigma \in S_N}
\sum_{A\subseteq [N], |A|=n}
\prod_{j\in A}\Big (1-\frac{\xx[j]}{p}\Big )\Big (1-\frac{\yy[\sigma(j)]}{p}\Big )
\nonumber \\
&&~=
\frac{1}{N!}\frac{(N-n)!}{n!}
\sum_{A\subseteq [N], |A|=n}\prod_{j\in A}\Big (1-\frac{\yy[j]}{p}\Big )
\times\sum_{A\subseteq [N], |A|=n}\prod_{j\in A}\Big (1-\frac{\xx[j]}{p}\Big )
\nonumber \\
&&~={N\choose n}Q_n(\|\bm{x}\|;N,p)Q_n(\|\bm{y}\|;N,p).
\end{eqnarray*}
Summing to find the distribution of $\|\bm{y}\|$ in the permutation distribution of $\bm{y}$ from (\ref{gen:10}) gives (\ref{H:10}).
\end{proof}

 %
\section{Proof of Theorem~\ref{th-low-t}.}
Fix $N$ and fix a coordinate $\ell$ which minimizes $\P(Z^{\ssup N}[\ell] =1)= \theta_N$.
Let $A =\{\yy \in \Vcal_N \colon \yy[\ell] =1\}$. Choose 
$t \le  a/\theta_N,$
where  $a$ is chosen as follows.
The quantity  $(1-\theta_N)^{1/\theta_N}$ is bounded away from $0$ as long as $\theta_N$ is bounded away from 1. Choose $a$ such that 
$$  (1-\theta_N)^{a/\theta_N}  > 1 - p_N/2.$$ 
We have that 
\begin{equation}\label{transi:bou}
 P_t(A\;|\; \0) = 1 \le 1 - (1-\theta_N)^{a/\theta_N} \le \frac {p_N}2.
 \end{equation}
We have that $\pi(A) =p_N$. Hence, 
$$
\|P_t(\cdot\;|\; \0) - \pi(\cdot, p_N)_N)\|_{TV} \ge \pi(A, p_N) - P_t(A\;|\; \0) > \frac {p_N}2 \ge \frac 14.
$$
\section{Proof of Theorem~\ref{3step}.}
\begin{tcolorbox}
\begin{definition}
Let $(H_n(v)\colon n \in \N, v \in \R)$  be the Hermite polynomials, which are defined through  the  generating function
\begin{equation}\label{eq:gen1}
\sum_{n=0}^\infty H_n(v)\frac{\psi^n}{n!} = e^{\psi v-\frac{1}{2}\psi^2}.
\end{equation}
\end{definition}
\end{tcolorbox}
\added{Notice that the $H_n(v)$ are orthogonal polynomials with respect to the standard normal distribution, i.e. for $n \neq m$, we have
$$ \int_{-\infty}^\infty H_n(v) H_m(v)  \frac 1{\sqrt{2 \pi}} {\rm e}^{- v^2/2} {\rm d} v = n!\delta_{mn},$$} where $\delta_{mn}$ is the Kronecker delta.
\added{
\begin{rmk} In what follows we use the following notation. For two sequences $(a_N)_N$ and $(b_N)_N$ of real numbers $a_N \sim b_N$ if an only if
$$ \lim_{N \ti} \frac {a_N}{b_N} =1.$$
\end{rmk}}
Recall that $q_N = 1- p_N$ and that $\lim_{N \ti} p_N = p = 1-q$.
\begin{propo}\label{pro:conv} Under the assumptions of Theorem~\ref{3step}
  then we have
\[
\lim_{N \ti} \E \big [Q_n(\zeta_N ;N,p_N)\big ] =\frac{(-1)^n}{ h_n^{1/2}  \sqrt{n!}}  \E [H_n(V)].
\]
\end{propo}
\begin{proof}
It is enough to prove convergence in distribution, as we can use the moment condition~\eqref{eq:mom} to appeal to the dominated convergence theorem. In turn,
in order  to prove the convergence in distribution, it is enough to prove that for any sequence   $z_N$  such that 
\begin{equation}
\lim_{N \ti}\frac{z_N  - Np_N}{\sqrt{Np_N(1- p_N)}}  = v
\end{equation}
 for some number $v$, we have that 
\begin{equation}\label{eq:detv}
\lim_{N \ti} h_n^{1/2}Q_n(z_N;N,p_N) =\frac{(-1)^n}{(n!)^{1/2}} H_n(v).
\end{equation}

We prove the convergence in \eqref{eq:detv} using generating function. 

Note that  for fixed $n$, as $N \ti$ with $p_N \to p$ we have
\[
(n!)^{1/2}h_n^{1/2} = \Big (n!{N\choose n}(p_N/q_N)^n\Big )^{1/2} \sim \Big (N(p/q)\Big )^{n/2}.
\]
Hence, we get the following estimate, which holds for all $z, N \in \N$ and $p \in [1/2, 1)$,
\begin{eqnarray}
\sum_{n=0}^N(n!)^{1/2}h_n^{1/2}Q_n(z;N,p_N) \frac{s^n}{n!} &\sim&
\sum_{n=0}^N{N\choose n}Q_n(z;N,p)\Big (\sqrt{\frac p{Nq}}s\Big )^n
\nonumber \\
&=&\Big (1-(q/p)\sqrt{\frac p{Nq}}s\Big )^z
\Big (1 + \big (\sqrt{\frac p{Nq}}s\Big )^{N-z}.
\nonumber \\
\label{zv:0}
\end{eqnarray}
\added{Taking the logarithm  of both sides of (\ref{zv:0})  and setting  $a=(q/p)\sqrt{(p/q)}s=\sqrt{(q/p)}s$, $b=\sqrt{(p/q)}s$, we have}
\begin{eqnarray}
\ln \sum_{n=0}^n(n!)^{1/2}h_n^{1/2}Q_n(z;N,p) \frac{s^n}{n!} &\sim&z\log \Big (1 - \frac{a}{\sqrt{N}}\Big )+(N-z)\log \Big (1 + \frac{b}{\sqrt{N}}\Big )
\nonumber \\
&&~= -z\Big (\frac{a+b}{\sqrt{N}}+\frac{1}{2}\frac{a^2 - b^2}{N}\Big ) + N\Big(\frac{b}{\sqrt{N}} -\frac{1}{2}\frac{b^2}{N}\Big ) + \mathcal{O}(N^{-1/2})
\nonumber \\
&=& -\sqrt{N}p(a+b) + \sqrt{N}b -v\sqrt{pq}(a+b) -\frac{1}{2}(a^2-b^2)p -\frac{1}{2}b^2
\nonumber \\
&&~~~~~~~~~~ - \frac{1}{2}(a^2-b^2)\sqrt{pq}vN^{-1/2}+\mathcal{O}(N^{-1/2}).
\label{temp:1000}
\end{eqnarray}
We have the following simplifications  in (\ref{temp:1000})
\begin{eqnarray*}
-p(a+b) + b &=& -pa + qb = -\sqrt{pq}s + \sqrt{pq}s = 0
\nonumber \\
-\sqrt{pq}(a+b) &=& -qs - ps = -s
\nonumber \\
-(a^2-b^2)p - b^2 &= &-\Big (\frac{q}{p}-\frac{p}{q}\Big )ps^2 - \frac{p}{q}s^2 = -s^2
\end{eqnarray*}
so (\ref{temp:1000}) is equal to 
\[
-vs - \frac{1}{2}s^2 + v\mathcal{O}(N^{-1/2}).
\]

That is, the generating function (\ref{zv:0}) is equal to
\[
\exp \big \{-vs - \frac{1}{2}s^2 + v\mathcal{O}(N^{-1/2})\big \}
\]
which converges to the generating function of $(-1)^nH_n(v)$.
\end{proof}
Using a C\'esaro sum argument, we immediately get from Proposition~\ref{pro:conv} the following  result.
\begin{cor}\label{cor:appr}
For all $p\in [0,1]$ and $t >0$,
\begin{equation}\label{HCC:0}
\lim_{N \ti} \frac{\sum_{n=1}^N{N\choose n}\Big (p_N/(1- p_N)\Big )^n\rho_n^{2t}
}{\sum_{n=1}^N\frac{1}{N^{n(t-1)}}\Big ((1- p_N)/p_N\Big )^{n(t-1)}\frac{1}{n!}(\E[H_n(V)])^{2t}} =1.
\end{equation}
\end{cor}
\begin{proof}[Proof of Theorem~\ref{3step}] It is well-known 
(e.g. see Lemma 12.16 in \cite{LPW2009})  that for a reversible Markov chain
\begin{equation}
\chi^2(\xx, t) = \sum_{n\geq 1}\lambda_n^{2t}f_n(x)^2
\end{equation}
where $\lambda_0=1$ and $\{\lambda_i\}_{i\geq 0}$ are eigenvalues and $f_j(x)$ orthonormal eigenvectors with respect to the stationary distribution.

  In our context
  \begin{eqnarray}
\chi^2(\xx, t)&= &\sum_{A\subseteq [N]\colon A \neq \emptyset}\rho_A^{2t} \prod_{i\in A}\Big (\frac{p_N-\xx[i]}{\sqrt{p_N q_N}}\Big )^2
\nonumber \\
&= &\sum_{A\subseteq [N]\colon A \neq \emptyset}\rho_A^{2t} \Big (\frac{p_N}{q_N}\Big )^{|A|} \prod_{i\in A}\Big (1 - \frac{\xx[i]}{p_N}\Big )^2\\
&\le& \sum_{n=1}^N{N\choose n}\Big (\frac{p_N}{q_N}\Big )^n\rho_n^{2t}.
\nonumber 
\label{better:10}
\end{eqnarray}
Notice that the bound is sharp, as it is achieved for initial condition $\xx = \0$. Hence,
$$ 
\max_{\xx} \chi^2(\xx, t) =\chi^2(\0, t) =  \sum_{n=1}^N{N\choose n}\Big (\frac{p_N}{q_N}\Big )^n\rho_n^{2t}.
$$
In virtue of Corollary \ref{cor:appr} in order to have an estimate of the $\chi^2$ distance, i.e. the numerator in the right-hand side of \eqref{HCC:0}, we simply need an estimate of the denominator, i.e.
\begin{equation}\label{eq:den}
\sum_{n=1}^N\frac{1}{N^{n(t-1)}}\Big (\frac{q_N}{p_N}\Big )^{n(t-1)}\frac{1}{n!}(\E[H_n(V)])^{2t}.
\end{equation}
To this end, we use the following well-known  formula (see, e.g.,  \cite{OLBC2010} 18.10.10 p448) which holds for any $v \in \R$,
\[
H_n(v) = \frac{2^{n+1}}{\sqrt{\pi}}e^{v^2/2}\int_0^\infty e^{-\tau^2}\tau^n\cos\big  (\sqrt{2}v\tau - \frac{1}{2}n\pi\big )d\tau.
\]
Therefore
\begin{eqnarray}
|H_n(v)| \leq   \frac{2^{n+1}}{\sqrt{\pi}}e^{v^2/2}\int_0^\infty e^{-\tau^2}\tau^nd\tau 
&=& \frac{2^{n+1}}{\sqrt{\pi}}e^{v^2/2}\frac{1}{2}\Gamma \Big (\frac{n}{2}+\frac{1}{2}\Big )
\nonumber \\
&=&
\begin{cases}
 e^{v^2/2}\frac{(2m)!}{2^mm!}& n=2m\\
 \\
\frac{2^{2m+1}}{\sqrt{\pi}}e^{v^2/2}m!&n=2m+1
\end{cases}.
\label{bound:100}
\end{eqnarray}
If $n$ is even $|H_n(v)| \leq e^{v^2/2}|H_n(0)|$.
We use these estimates to provide bounds for the sum of 
 even terms in \eqref{eq:den} as follows
\begin{equation}\label{bound-H1}
\begin{aligned}
&\sum_{m=1}^{\floor{N/2}}\frac{1}{N^{2m(t-1)}}\Big (\frac{q_N}{p_N}\Big )^{2m(t-1)}\frac{1}{(2m)!}(\E[H_{2m}(V)])^{2t}\\
&\; \le  \E[e^{V^2/2}]^{2t} \sum_{m=1}^{[N/2]}\frac{1}{N^{2m(t-1)}}\Big (\frac{q_N}{p_N}\Big )^{2m(t-1)}\frac{1}{(2m)!}
\Big (\frac{(2m)!}{2^mm!}\Big )^{2t}.
\end{aligned}
\end{equation}
Denote  the terms in the sum in the right-hand side of  \eqref{bound-H1} as $b_m$,
\[
\frac{b_{m+1}}{b_m} = \frac{1}{N}\cdot \frac{1}{2(m+1)}\Big (\frac{q_N}{p_N}\Big )^{2(t-1)}\Big (\frac{2m+1}{N}\Big )^{2t-1}<1
\]
for $m+1 \leq [N/2]$. Hence, $\max_{m \le [N/2]} b_m = b_1$, i.e. the first term in the sum. Therefore 
\[
\sum_{m=1}^{\floor{N/2}}\frac{1}{N^{2m(t-1)}}\Big (\frac{q_N}{p_N}\Big )^{2m(t-1)}\frac{1}{(2m)!}(\E[H_{2m}(V)])^{2t} \le  \E[e^{V^2/2}]^{2t}  
\frac{N}{2}\frac{1}{\sqrt{\pi}}\frac{1}{N^{2(t-1)}}\Big (\frac{q_N}{p_N}\Big )^{2(t-1)}\frac{1}{2^{2t}}
\]
which tends to zero as $N\to \infty$ if $t>3/2$.
Notice that $H_{2n+1}(0) = 0$. Hence, in the case of  $\P(V= 0) =1$ we do not have to estimate the odd terms, and the chain mixes almost-perfectly in 2 steps. When $\P(V= 0) < 1$
we  consider the odd terms, 
\[
\begin{aligned}
&\sum_{m=0}^{\floor{N/2}}\frac{1}{N^{(2m+1)(t-1)}}\Big (\frac{q_N}{p_N}\Big )^{(2m+1)(t-1)}\frac{1}{(2m+1)!}(\E[H_{2m+1}(V)])^{2t}\\ 
&\le \frac{1}{\sqrt{\pi}} \E[e^{V^2/2}]^{2t} \sum_{m=0}^{\floor{N/2}}\frac{1}{N^{(2m+1)(t-1)}}\Big (\frac{2q_N}{p_N}\Big )^{(2m+1)(t-1)}\frac{1}{(2m+1)!}m!^{2t}.
\end{aligned}
\]
Writing the terms in the latest sum as $c_m$, we have
\[
\frac{c_{m+1}}{c_m} = \Big (\frac{2q_N}{p_N}\Big )^{2(t-1)}\frac{1}{2m+3}\Big (\frac{m+1}{N}\Big )^{2t-1}.
\]
If $t$ is of constant order this ratio is less than one for $N$ sufficiently large because there exists a $m_0$ of constant order such that for $m\geq m_0$,
 $2^{2(t-1)}/(2m+3) < 1$, then an $N_0$ can be chosen such that for $N \geq N_0$ the terms in the ratio for $m < m_0$ are less than 1.
The first term is then again maximal for $N$ large enough with $t$ of constant order, and the sum of the odd terms is less than
\[
\frac{N}{2}\frac{1}{\sqrt{\pi}} \E[e^{V^2/2}]^{2t} \frac{2^{t-1}}{N^{t-1}} \Big (\frac{q_N}{p_N}\Big )^{t-1}
\]
which tends to zero if $t>2$. Taking into account both even and odd terms in \eqref{eq:den} the mixing time is $t=3$ if $\mathbb{P}(V=0) < 1$. 
\end{proof}

\section{Proof of Theorems~\ref{thm5} }
\subsection{\added{Lower bound for $t_{mix}$}}
\noindent The following Theorem is due to David Wilson (see, e.g., Theorem 13.5, p172 in \cite{LPW2009}).
\begin{thm}\label{th:Wilson}{\rm \bf [Wilson bound]}
Let $\X$ be an irreducible aperiodic Markov chain with state space $\Omega$. Let $\Phi $ be an eigenfunction with eigenvalue $\lambda$ satisfying $1/2 < \lambda <1$. Fix $0 < \epsilon <1$ and let $R>0$ satisfy
\[
\mathbb{E}_{\xx}\Big [\big|\Phi (X_1) - \Phi(\xx)\big|^2\Big ] \leq R
\]
for all $x\in \Omega$. Then for any $x \in \Omega$
\begin{equation}\label{eq:Wils}
t_{\text{mix}}(\epsilon) \geq \frac{1}{2\log (1/\lambda)}
\Bigg [\log \Bigg ( \frac{(1-\lambda)\Phi(\xx)^2}{2R} \Bigg )+ \log \Bigg (\frac{1-\epsilon}{\epsilon}\Bigg )\Bigg ].
\end{equation}
\end{thm}
\added{Next, consider a sequence $(\X^{\ssup N})_{N\in \mathbb{N}} \in \mathcal{C}$. We apply Wilson\rq{}s Lemma to  each element of the sequence, with the choice of first eigenvalue and eigenvector pair. Then}
\[
\Phi_N(\xx) = \|\xx\|-Np,\> \lambda_N = 1 -\frac {z_N}{Np}.
\]
From (\ref{expansion:2}) in the Appendix, 
\begin{eqnarray*}
\mathbb{E}_{\xx}\big [X_1(X_1-1) + X_1\big ] &=& N(N-1)p^2\Big (\rho_2Q_2(\xx; N, p_N) - 2\rho_1Q_1(\xx; N, p_N) + 1\Big )\\
 &+& Np\Big (-\rho_1Q_1(\xx; N, p_N) +1)
\end{eqnarray*}
In particular
\begin{eqnarray}
\mathbb{E}_0\big [X_1^2] &\sim& N(N-1)p^2\Big ( \big ( 1 - \frac{z_N}{Np_N}\big )^2 - 2\big (1 - \frac{z_N}{Np_N}\big ) +1\Big ) + Np_N\Big (-\big (1-\frac{z_N}{Np_N}\big ) +1\Big )
\nonumber \\
&=& \Big (1 - \frac{1}{N}\Big )z^2_N + z_N
\end{eqnarray}

$1/2 < \lambda_N < 1$ is satisfied if $z_N/N < p_N/2$.
We fix $\xx = (0, 0, \ldots, 0) \in \Vcal_N$. Notice that we cover also the case $\lim_{N \ti} z_N/N = 0$. There exists a constant $c>0$ such that 
 \begin{equation}\label{eq:logb}
 \log \lambda_N \ge - c \frac{z_N}{Np_N}.
 \end{equation}
  Using \eqref{eq:logb}  in Wilson's bound \eqref{eq:Wils}, we get
\begin{eqnarray}
t_{\text{mix}}(\epsilon) &\geq&  c \frac{Np_N}{2z_N}\Bigg [\log \Bigg (\frac{z_N}{Np_N}\frac{(Np_N)^2}{2R}\Bigg )+\log \Bigg (\frac{1-\epsilon}{\epsilon}\Bigg )\Bigg ]
\nonumber \\
&\geq&  c \frac{Np_N}{2z_N}\Bigg [\log\Bigg ( \frac{Np_Nz_N}{2R}\Bigg )+\log \Bigg (\frac{1-\epsilon}{\epsilon}\Bigg )\Bigg ]
\nonumber \\
&\geq &  c \frac{Np_N}{2z_N}\Bigg [\log N + \log \frac{p_N}{2R} + \log \frac{1-\epsilon}{\epsilon}\Bigg ].
\label{lower:a}
\end{eqnarray}
The dominant term in (\ref{lower:a}) together gives that 
\begin{equation}\label{eq:mix}
t_{\text{mix}}(\epsilon) \ge  c  \frac{Np_N}{2z_N}\log N + \mathcal{O}(N).
\end{equation}
Notice that the bound in equation \eqref{eq:logb} is required just for all large $N$. Hence, in the case $\lim_{N \ti} z_N/N = 0$, we can choose any $c \in (0,1)$. Hence, for any $\eps>0$, we have 
\begin{equation}\label{eq:mix1}
t_{\text{mix}}(\epsilon) \ge (1- \eps)  \frac{Np_N}{2z_N}\log N + \mathcal{O}(N).
\end{equation}

%
\subsection*{\added{Comparison between $t_{mix}$ and  $t^{\ssup 2}_{mix}$.}  }
\added{The following proposition is well-known in the literature. We add a proof here for the sake of clarity and completeness.
\begin{propo}\label{rmkdist} We will have the following useful relation between the  total variation  and the $\chi^2$ distances
 $$ 4 \|P_t(\xx, \cdot)- \pi_N(\cdot)\|_{TV}^2 \le \chi^2(\xx, t), \qquad \mbox{for all $\xx \in \Omega$}.$$
\end{propo}}
\begin{proof}
\added{
\begin{eqnarray*}
||P_t(\cdot \mid x) - \pi_N(\cdot)||_{\text{TV}} &=& \frac{1}{2} \sum_y|P_t(y\mid x) - \pi_N(y)|
\\
& = & \frac{1}{2}\sum_y \sqrt{\pi_N(y)}|P_t(y\mid x) - \pi_N(y)|/\sqrt{\pi_N(y)}
\nonumber \\
& \leq & \frac{1}{2}\sqrt{\chi^2(x)},
\end{eqnarray*}
where in the last step, we used Cauchy-Schwartz inequality.}
\end{proof}
A by-product we have the following corollary.
\begin{cor}\label{compa}
 \added{$t_{mix}(\eps) \le t_{mix}^{\ssup 2}(\sqrt{\eps}/4).$}
\end{cor}
\subsection{\added{Digression on which $\chi^2$-distance to use.}}\label{se:digr}
Recall that
\begin{equation}\label{rhoz}
\rho_n=\mathbb{E}\Big [Q_n(\|Z_1\|;N,p)\Big ]
\end{equation}
then
for sequences;
\begin{equation}
\chi^2_t(\bm{x}) = \sum_{n\geq 1}h_n
{N\choose n}^{-1}\sum_{A\subseteq [N],  |A|=n}
\Bigg (\mathbb{E}\Big [\prod_{j\in A}\Big (1- \frac{Z_1[j]}{p}\Big )\Big ]\Bigg )^{2t}
\prod_{j\in A}\Big (1- \frac{\xx[j]}{p}\Big )^2.
\label{chi:AA}
\end{equation}
From now on, in this section, assume that  $Z$ is exchangeable.
The $\chi^2_t$ distance (\ref{chi:AA}) when $Z$ is exchangeable simplifies to
\begin{equation}
\chi^2_t(\bm{x}) = \sum_{n\geq 1}h_n
\rho_n^{2t}
{N\choose n}^{-1}
\sum_{A\subseteq [N], |A|=n}
\prod_{j\in A}\Big (1- \frac{\xx[j]}{p}\Big )^2.
\label{chi:BB}
\end{equation}
\added{On the other hand, recall that  $\overline{P}_t(\cdot\;| \xx)$ the p.m.f. of $\|X_t\|$ conditional to $X_0 = \xx$. Let $\mathbb{Q}_N$ be a Binomial with parameters $N$ and $p$. Define $\chi^2$ for the Hamming distance as $\chi_H^2(\bm{x}, t) =\chi^2(\overline{P}_t(\cdot\;| \xx)\;|\; \mathbb{Q}_N)$. We have  the following representation}
\begin{equation}\label{hamd2}
\chi_H^2(\bm{x},t) = \sum_{n\geq 1}h_n
\rho_n^{2t}
Q_n(\|\xx\|;N,p)^2.
\end{equation}
In general, we have  that
\[
\chi^2_t(\bm{x})\geq \chi^2_{H}(\bm{x}, t)
\]
which accords with intuition. This is because
\begin{eqnarray*}
{N\choose n}^{-1}\sum_{A\subset [N], |A|=n}
\prod_{j\in A}\Big (1- \frac{\xx[j]}{p}\Big )^2
&\geq& 
\Bigg ({N\choose n}^{-1}\sum_{A\subset [N], |A|=n}
\prod_{j\in A}\Big (1- \frac{\xx[j]}{p}\Big )\Bigg )^2
\nonumber \\
&=&
Q_n(\|\xx\|,N,p)^2
\end{eqnarray*}
On the other hand, we already proved that  
the supremum over $\bm{x}$ of the two distinct $\chi^2$'s distances  (\ref{chi:BB}) and (\ref{hamd2}) coincide, and occurs when $\bm{x}=\bm{0}$. In other words, $\sup_{\xx} \chi^2_{H}(t, \bm{x}) = \chi^2_{H}(t, \bm{0})$ and
\begin{equation}\label{eq:hamvs}
\sup_{\xx \in \Vcal_N} \chi^2(t, \bm{x}) = \chi^2(t, \bm{0})=\chi^2_{H}(t, \bm{0})= 
\sum_{n\geq 1}h_n
\rho_n^{2t}.
\end{equation}
The reasoning above implies that if we look at the worse-case scenario, in terms of initial configurations, $\chi^2$ and $\chi^2_H$ behave in the same way. On the other hand, it is possible to choose initial conditions that make $\chi^2_H$ much smaller than $\chi^2$, resulting in a faster mixing for the Hamming distance. This gives somehow the intuition behind Theorem~\ref{thm6}.

\subsection*{\added{Upper bound for $t^{\ssup 2}_{mix}$.}  }
 
We assume that $\|\xx\| \neq Np_N.$

\added{In virtue of our reasoning in the previous section, we can use the $\chi^2$ distance for the Hamming distance, as when we take the supremum over $\xx$, it coincides with $\chi_t^2$. Recall that}
\[
\chi^2_H(\xx, t) = \sum_{n=1}^N\rho_n^{2t}h_nQ_n(\|\xx\|;N,p_N)^2
\]
\begin{propo}
\added{Under the assumptions of Theorem~\ref{thm5}, we have that  for any $x \in [0, N] \cap \N$, 
\begin{equation}\label{appr:q}
Q_n(x, N,p_N) \sim \left(1-\frac x{Np_N}\right)^n.
\end{equation}}
 \end{propo}
 \begin{proof}
 Replace $s$ by $s/N$ in the generating function (\ref{genfn:0}) and take the  logarithm of both sides to get 
 \begin{equation}\label{eq:mist1}
 \log \left(\sum_{n=0}^N{N\choose n}Q_n(x;N,p_N)\frac{s^n}{N^n}\right) = x \log \left(1-\frac{qs}{p_N N}\right) + (N-x) \log \left(1+\frac sN\right).
 \end{equation}

 Let $\zeta_N = x/N$ and $\alpha = q_N/p_N$.  The right-hand side of \eqref{eq:mist1} becomes
\begin{eqnarray}
&&N\zeta_N\log (1-s\alpha_N/N) + N(1-\zeta_N) \log (1+s/N) 
\nonumber \\
&&= \zeta_N \Big ( -\alpha_N s - \frac{1}{2N}\alpha^2_Ns^2\Big ) + (1-\zeta_N) \Big ( s - \frac{1}{2N}s^2\Big ) + \mathcal{O}(N^{-2})
\nonumber \\
&&=  s - \frac{1}{2N}s^2 - \zeta_N\Big (s/p_N - s^2(p_N-q_N)\frac{1}{2Np_N}\Big ) + \mathcal{O}(N^{-2}).
\label{asy:100}
\end{eqnarray}
If $\zeta_N\ne p$ then (\ref{asy:100}) is equal to
\begin{equation*}
 s(1-\zeta_N/p_N) +  \mathcal{O}(N^{-1})
\end{equation*}
however if $\zeta_N=p_N$ then (\ref{asy:100}) is equal to
\[
-(q_N/p_N)s^2\frac{1}{2N} +  \mathcal{O}(N^{-2}).
\]
Therefore for fixed $\zeta_N\ne p$, where recall that $p = \lim_{N \ti} p_N$, asymptotic values are
\[
Q_n(N\zeta_N;N,p_N) \sim \Big (1-\zeta_N/p \Big)^n.
\]
If $\zeta_N=p$ then $Q_{2n+1}(Np;N,p) = 0$ and
\[
Q_{2n}(Np;N,p) = (-q/p)^n\frac{(2n)!}{n!}\frac{1}{(2N)^n}.
\]
If $p=q$ then from the original generating function of $(1-s)^{N/2}(1+s)^{N/2}$ 
\[
Q_{2n}(N/2;N,p) = (-1)^n\frac{{N/2\choose n}}{{N\choose 2n}},
\]
which agrees with the case above with $p\to 1/2$ as $N \to \infty$.
 \end{proof}
\bigskip

\noindent
\added{Combining \eqref{rhoz} with Proposition~\ref{appr:q}, we have  that} 
\begin{equation}\label{rhoz2} \rho_n \sim \left(1- \frac {z_N}{Np}\right)^n.\end{equation}
Hence, using \eqref{hamd2}, we have 
\begin{eqnarray}
&&\chi^2_H(\xx) \sim \sum_{n=1}^N\Big (1 - \frac{z_N}{Np}\Big )^{2nt}{N\choose n}\Big (\frac{p}{q}\Big )^n\left(1-\frac{\|\xx\|}{Np}\right)^{2n}
\nonumber \\
&&=\sum_{n=1}^N\frac{N^n}{n!}\Big (1 - \frac{z_N}{Np}\Big )^{2nt}\Big (\frac{p}{q}\Big )^n\Big (1-\frac{\|\xx\|}{Np}\Big )^{2n}
\nonumber \\
&&\leq \exp \Bigg \{ N\Big (1-\frac{z_N}{Np}\Big )^{2t}\Big (1-\frac{\|\xx\|}{Np}\Big )^{2}\Big (\frac{p}{q}\Big ) \Bigg \}-1
\nonumber \\
&&\leq \exp \Bigg \{ N e^{-2t\frac{z_N}{Np}}\Big (1-\frac{\|\xx\|}{Np}\Big )^{2}\Big (\frac{p}{q}\Big ) \Bigg \}-1\nonumber  \\
&&\leq \exp \Bigg \{ N e^{-2t\frac{z_N}{Np}}\Big (\frac{p}{q}\Big ) \Bigg \}-1.
\label{chi:0}
\end{eqnarray}
Choose 
\begin{equation}
t_N=\frac{Np}{2z_N}(\log N + C)
\label{cutoff:0}
\end{equation}
 then the upper bound in (\ref{chi:0}) is equal to
\begin{equation}
\exp\Bigg \{\exp\{-C\}\Big (\frac{p}{q}\Big )\Bigg \}-1.
\label{upper:0}
\end{equation}
For a lower bound take the first term  in the $\chi_t^2(\0)$ expression.
\begin{eqnarray}
\chi^2_{t_N} \ge \chi^2_{t_N}(\0) = \chi^2_H(\0, t_N) &\geq& \Big (1 - \frac{z_N}{Np}\Big )^{2t_N}N\Big (\frac{p}{q}\Big )
\nonumber \\
&\to& e^{-C}\Big (\frac{p}{q}\Big ).
\label{lower:0}
\end{eqnarray}
This shows that (\ref{cutoff:0}) is a cutoff time because if $C$ is large and positive both the upper and lower bounds are small,  
and if $C$ is large and negative both bounds are large. 
\added{Calculations here are related to chi-squared cutoff calculations for a multinomial model in \cite{DG2019}, Section 4.1.}
\section{Further discussion and a few examples}
{
\begin{Ex} {\bf Contingency Table.} In a classical paper \cite{AG1935} derive the joint marginal distribution of $1$'s in rows and columns of a $2\times 2$ contingency table with categories $0$ and $1$.  
Let $P$ be $2\times 2$ bivariate probability matrix describing probabilities in the contingency table.
Fix the margins in $P$ to be Bernoulli $(p)$. Take $N$ independent observations in the table to form a table $\bm{n}=(n_{ij};i,j\in \{0,1\})$. Let $X=n_{10}+n_{11}$ and $Y=n_{01}+n_{11}$. Then $Y\mid X$ has a p.m.f. of the form (\ref{H:10}) (with notation $X=\|\bm{x}\|, Y = \|\bm{y}\|$) where $\rho_n = \rho^n$, such that $\rho=(p_{11}-p^2)/pq$ is the correlation coefficient in $P$.
The conditional distribution fits into Example \ref{ex:i.i.d} where each coordinate is updated independently with probability $\alpha_{N,t}=p_{01}/q$. 
The $\chi^2$ cutoff time when $\rho\ne 0$ is 
\[
t_N = \frac{\log N + \log (p/q)+ C}{-2\log |\rho|}.
\]
\end{Ex}
The proof is sketched as follows. The spectral representation for $\chi^2_{t_N}$ when starting from $\0$ is  
\begin{eqnarray}
\sum_{n=1}^Nh_n\rho^{2nt_N} &=& \sum_{n=1}^N{N\choose n}\Big (\frac{p}{q}\Big )^n\rho^{2nt_N}
\nonumber \\
&=& \Big (1+\frac{p}{q}\rho^{2t_N}\Big )^N-1
\nonumber \\
&=& \Big (1 + e^{\log (p/q)+ 2t_N\log |\rho|}\Big )^N - 1
\nonumber \\
&\to& e^{e^{-C}}-1,
\label{temp:A}
\end{eqnarray}
showing that $t_N$ is a cutoff time by taking $C>0$ or $C<0$ and $|C|$ large.
}

\begin{Ex}
This example illustrates the difference between the chi-squared cutoff for the Hamming distance and the chi-squared cutoff for the sequences depending on the initial $\bm{x}$. 
Consider a model where $||Z||=N$ and $||\bm{x}||=Np$.
The factor in $\chi^2_{t}(\bm{x})$ of
\[
{N\choose n}^{-1}\sum_{A\subset [N], |A|=n} \Big (1 - \frac{\bm{x}[i]}{p}\Big )^2.
\]
is the coefficient of ${N\choose n}s^n$ in the generating function
\[
\Big (1 + \Big (\frac{q}{p}\Big )^2s\Big )^{Np}\Big (1 + s\Big )^{Nq}.
\]
Replacing $s$ by $s/N$
\[
\Big (1 + \Big (\frac{q}{p}\Big )^2\frac{s}{N}\Big )^{Np}\Big (1 + \frac{s}{N}\Big )^{Nq}
\to \exp \Big \{ s\frac{q}{p}\Big \}.
\]
Then it follows that
\[
\chi^2_{t}(\bm{x}) \sim \sum_{n=1}^N\rho_n^{2t}h_n\Big (\frac{q}{p}\Big )^n.
\]
If $||Z||=N$ then $\rho_n=(-q/p)^n$ and
\[
\chi^2_{t}(\bm{x}) \sim \sum_{n=1}^N\Big (\frac{q}{p}\Big )^{2nt}{N\choose n}
\Big (\frac{p}{q}\Big )^n\Big (\frac{q}{p}\Big )^n
 =
\Bigg ( 1 + \Big (\frac{q}{p}\Big )^{2t}\Bigg )^N - 1.
\]
A calculation shows then, with 
\begin{equation}
t_N = \frac{\log N + C}{-2\log (q/p)}
\label{xyz:000}
\end{equation}
then
\[
e^{-C} < \chi^2_{t}(\bm{x}) < \exp\{e^{-C}\}-1
\]
so
the cutoff time calculated from $\chi^2_{t}(\bm{x})$ is given by (\ref{xyz:000})
compared to the Hamming distance which has a finite mixing time when $w=p\ne 1/2$.
\end{Ex}

\section{Proof of Theorem~\ref{thm6}}
Let 
$Q_{2n+1}(Np;N,p) = 0$ and
$
Q_{2n}(Np;N,p) = (-q/p)^n\frac{(2n)!}{n!}\frac{1}{(2N)^n}.
$
Then
\begin{equation}
\chi^2_H(Np, t) = \sum_{n=1}^{[N/2]}Q_{2n}(z;N,p)^{2t}{N\choose n}\Big (\frac{p}{q}\Big )^n\Bigg (\frac{(2n)!}{n!}\frac{1}{(2N)^n}\Bigg )^2.
\end{equation}
If $z/N=w$ and $w\ne p$, using Proposition~\ref{appr:q}  we have
\begin{eqnarray}
\chi^2_H(Np, t) \sim \sum_{n=1}^{[N/2]}\left(1-\frac{w}{p}\right)^{2nt}{N\choose n}\Big (\frac{p}{q}\Big )^n\Bigg (\frac{(2n)!}{n!}\frac{1}{(2N)^n}\Bigg )^2.
\label{sum:1000}
\end{eqnarray}
Let $b_n$ denote the nth term in the sum (\ref{sum:1000}). The ratio of terms
is
$$
\begin{aligned}
\frac{b_{n+1}}{b_n}&=\big (1-\frac{w}{p}\big )^{2t} \frac{N-n}{n+1}\frac{p}{q}\frac{(2n+1)^2}{N^2} <
\big (1-\frac{w}{p}\big )^{2t}  \frac{p}{q}  \frac{2(N+1)}{N} < 1,
\end{aligned}
$$
\added{for $t > t_\circ$, where $t_\circ$ is a finite time not depending on $N$.
The first term $b_1$ is therefore maximal for $t > t_\circ$ and 
\begin{equation}\label{boh1}
\chi^2_H(Np) < \frac{N}{2}b_1= \frac{1}{2}\big (1-\frac{w}{p}\big )^{2t}\frac{p}{q}.
\end{equation}
Choose $t_\eps > t_\circ$ such that the right-hand side of \eqref{boh1} is less than $\eps$. }

\section{Remarks and Conclusion}
\begin{rmk}
Let ${\mathcal Q}$ be the set of probability measures $\mathbf{P}$ such that there exists $\X \in \mathcal{G}$ satisfying
$$ \mathbf{P}(\cdot) = \P(X_{1} \in \cdot\;|\; X_{0} = \xx),$$
for some $\xx \in \Vcal_N$. The set 
 ${\mathcal Q}$ is convex. The extreme points of ${\mathcal Q}$ are the  processes in $\mathcal{G}$  where each of $(Z_t)_{t\in \mathbb{N}}$ take a single value with probability 1.
	Let ${\mathcal M}$ be the set of $2\times 2$ probability transition matrices with stationary distribution $(q,p)$. ${\mathcal M}$ is a convex set with extreme points $I$, the identity matrix, and the transition matrix {for acceptance/rejection of 
$$
\begin{pmatrix}
0&1\\
\frac{q}{p}&1-\frac{q}{p}
\end{pmatrix}.
$$
}
It is straightforward to show that a model where changes occur at coordinates according to $P_1,\ldots,P_N$, once picked to change with probability $Z_t$, can be expressed in terms of our model with a modified distribution for $Z_t$ by using extreme point representations for the transition matrices.
\end{rmk} 
\begin{rmk}
 The sufficiency that $\rho_A = \mathbb{E}\big [\prod_{i\in A}\big (1 - Z[i]/p)\big ]$ could also be proved by using an important hypergroup property which can be described as follows. There exists a Bernoulli random variable $\xi$ such that for $u,v \in \{0,1\}$,
$
\big (1 - u/p\big )\big (1 - v/p\big ) = \mathbb{E}_{uv}\big [1 - \xi/p].
$
In fact $P\big (\xi =1 \big ) =  \delta_{u+v,1} +  \Big (1-\frac{q}{p}\Big )\delta_{u+v,2}$. A general reference is \cite{BH2008} and a reference more in the context of this paper is \cite{DG2012}.
\end{rmk}
\section{Appendix}

\begin{propo} For any $x \in [N] \cup \{0\}$ we have that 
\begin{equation}
\begin{aligned}
x(x-1) &=  2q^2h_2Q_2(x; N, p) -2pq(N-1)h_1Q_1(x; N, p) + N(N-1)p^2\\
&=N(N-1)p^2Q_2(x; N, p) -2N(N-1)p^2Q_1(x; N, p) + N(N-1)p^2.
\label{expansion:2}
\end{aligned}
\end{equation}
\end{propo}
\begin{proof}
Consider, with expectation in the Binomial $(N,p)$ distribution 
\begin{eqnarray}
&&\mathbb{E}\big [X(X-1)\big ( 1 - \frac{q}{p}s\big )^X\big (1 + s\big )^{N-x}\big ]
\nonumber \\
&&= N(N-1)p^2\big ( 1 - \frac{q}{p}s\big )^2\Big (p\big ( 1 - \frac{q}{p}s\big )+q(1+s)\Big )^{N-2}
\nonumber \\
&&= N(N-1)p^2\big ( 1 - \frac{q}{p}s\big )^2.
\label{calc:20}
\end{eqnarray}
Looking at coefficients of $s$ and $s^2$,
\begin{eqnarray*}
\mathbb{E}\big [X(X-1)Q_1(X; N, p)\big ] &=& {N\choose 1}^{-1}N(N-1)p^2\times -2 \frac{q}{p} = -2pq(N-1)
\nonumber \\
\mathbb{E}\big [X(X-1)Q_2(X; N, p)\big ] &=&{N\choose 2}^{-1}N(N-1)p^2\times \Big (\frac{q}{p}\Big )^2 = 2q^2
\nonumber \\
\mathbb{E}\big [X(X-1)\big ] &=& N(N-1)p^2,
\end{eqnarray*}
which proves our result.
\end{proof}
\noindent {\bf Acknowledgements} We thank Persi Diaconis and  Tim Garoni for helpful comments that really improved our results.

\end{document}